\documentclass[reqno,10pt,a4paper]{amsart}
\usepackage{etex}
\usepackage{amsthm}
\usepackage{cite}
\usepackage[all]{xy}
\usepackage[margin=2.7cm]{geometry}
\usepackage{graphicx}
\usepackage[usenames,dvipsnames]{color}
\usepackage{hyperref}
\hypersetup{
 colorlinks,
 linkcolor={red!50!black},
 citecolor={blue!50!black},
 urlcolor={blue!80!black}
}
\usepackage{mleftright}
\usepackage{mathabx}
\usepackage{longtable}
\usepackage{booktabs}
\usepackage{colortbl}
\usepackage{amssymb}
\usepackage{longtable}
\usepackage{tikz}
\usepackage{mathtools}
\usepackage[font=small]{caption}
\usepackage{tikz}
\usetikzlibrary{shapes.misc}
\usepackage[inline]{enumitem}
\usepackage{float}
\usepackage{pgfplots}
\usepackage{subfigure}
\usepackage[norelsize,ruled,vlined]{algorithm2e}

\setlist[enumerate]{labelsep=*, leftmargin=1.5pc}
\setlist[enumerate]{label=\normalfont(\roman*), ref=\roman*}
\AtBeginDocument{%
\def\MR#1{}
}
\newtheorem{thm}{Theorem}[section]
\newtheorem{lemma}[thm]{Lemma}
\newtheorem{cor}[thm]{Corollary}
\newtheorem{prop}[thm]{Proposition}
\theoremstyle{definition}

\newtheorem{remark}[thm]{Remark}
\newtheorem{definition}[thm]{Definition}
\newtheorem{conjecture}[thm]{Conjecture}
\newtheorem{question}[thm]{Question}
\numberwithin{equation}{section}
\newcommand{\Z}{\mathbb{Z}}
\newcommand{\Q}{\mathbb{Q}}
\newcommand{\R}{\mathbb{R}}
\DeclareMathOperator{\GL}{GL}
\DeclareMathOperator{\Sym}{Sym}
\newcommand{\abs}[1]{\vert{#1}\vert}
\renewcommand{\gcd}[1]{\operatorname{gcd}\mleft\{{#1}\mright\}}

\DeclareMathOperator{\ehr}{L}
\newcommand{\orig}{\boldsymbol{0}}
\renewcommand{\o}[1]{o_{#1}}
\newcommand{\vol}[1]{\operatorname{vol}\mleft({#1}\mright)}
\newcommand{\bdry}[1]{\partial{#1}}
\newcommand{\intr}[1]{{#1}^\circ}
\newcommand{\Vol}[1]{\operatorname{Vol}\mleft({#1}\mright)}
\newcommand{\V}[1]{\operatorname{vert}\mleft({#1}\mright)}
\renewcommand{\dim}[1]{\operatorname{dim}\mleft({#1}\mright)}
\newcommand{\sconv}[1]{\operatorname{conv}\mleft\{{#1}\mright\}}
\newcommand{\conv}[1]{\operatorname{conv}\mleft({#1}\mright)}
\newcommand{\cA}{\mathcal{A}}
\newcommand{\cB}{\mathcal{B}}
\newcommand{\cP}{\mathcal{P}}
\newcommand{\cR}{\mathcal{R}}
\newcommand{\cS}{\mathcal{S}}
\newcommand{\cW}{\mathcal{W}}
\newcommand{\ulambda}{\underline{\lambda}}
\newcommand{\h}{\delta}
\newcommand{\otherh}{h^*}
\newcommand{\oddrow}{\rowcolor[gray]{0.95}}
\newcommand{\evnrow}{}
\begin{document}
\author[G.\,Balletti]{Gabriele Balletti}
\address{Department of Mathematics\\Stockholm University\\SE-$106$\ $91$\ Stockholm\\Sweden}
\email{balletti@math.su.se}
\author[A.\,M.\,Kasprzyk]{Alexander M.~Kasprzyk}
\address{School of Mathematical Sciences\\University of Nottingham\\Nottingham\\NG7 2RD\\UK}
\email{a.m.kasprzyk@nottingham.ac.uk}
\title{Three-dimensional lattice polytopes with two interior lattice points}
\begin{abstract}
We classify the three-dimensional lattice polytopes with two interior lattice points. Up to unimodular equivalence there are $22,\!673,\!449$ such polytopes. This classification allows us to verify, for this case only, a conjectural upper bound for the volume of a lattice polytope with interior points, and provides strong evidence for new conjectural inequalities on the coefficients of the Ehrhart $\h$-polynomial in dimension three.
\end{abstract}
\maketitle
\vspace{-2.2em}
\section{Introduction}\label{sec:intro}
We begin by fixing our notation. A \emph{lattice polytope} $P \subset\Z^d\otimes_\Z\R\cong\R^d$ is the convex hull of finitely many points in $\Z^d$. The set of vertices of $P$ is denoted by $\V{P}$, the relative boundary by $\bdry{P}$, and the relative (strict) interior by $\intr{P}$. The polytope $P$ is said to be a \emph{$k$-point polytope} if $\abs{\intr{P}\cap\Z^d}=k$. A $0$-point polytope is called \emph{hollow}. Two lattice polytopes $P,Q\subset \R^d$ are said to be \emph{unimodular equivalent} if there exists an affine lattice automorphism $\varphi\in\GL_d(\Z)\ltimes\Z^d$ of $\Z^d$ such that $\varphi_\R(P)=Q$, and we will typically consider polytopes only up to unimodular equivalence. Although a lattice polytope $P$ need not be of full dimension, we will assume that $\dim{P}=d$ unless stated otherwise; if we need to emphasise that $\dim{P}=n$ we will refer to $P$ as an $n$-polytope.

Much work has focused on developing explicit classifications of lattice polytopes. Perhaps the most celebrated example is the classification of all $473,\!800,\!776$ four-dimensional reflexive polytopes (a special class of $1$-point lattice polytopes) by Kreuzer--Skarke~\cite{KS00}. This classification was motivated by applications in theoretical physics and string theory, and to the study of smooth Calabi--Yau manifolds. Motivated by questions in algebraic geometry, all $674,\!688$ three-dimensional $1$-point lattice polytopes were classified in~\cite{Kas10}; these polytopes correspond to the toric Fano threefolds having at worst canonical singularities. Another special class of $1$-point lattice polytopes important in toric geometry are the smooth Fano polytopes. These have been classified up to dimension nine: in dimension three by Batyrev~\cite{Bat81} and Watanabe--Watanabe~\cite{WW82}, in dimension four by Batyrev~\cite{Bat99} and Sato~\cite{Sat00}, in dimension five by Kreuzer--Nill~\cite{KN09}, and finally an efficient algorithm for arbitrary dimensions was described by {\O}bro~\cite{Obr07}. More recently, Blanco--Santos have classified those three-dimensional lattice polytopes with $\abs{P\cap\Z^3}\leq 11$ and lattice width at least two~\cite{BS16_1,BS16_2,BS16_3}.

The $2$-point lattice polygons (that is, the two-dimensional polytopes with $\abs{\intr{P}\cap\Z^2}=2$) were classified by Wei--Ding~\cite{WD12}. In this paper we derive a complete classification of all $2$-point three-dimensional polytopes; the classification is summarised in the following theorem:

\begin{thm}\label{thm:main}
Up to unimodular equivalence there are exactly $22,\!673,\!449$ $2$-point three-dimensional lattice polytopes. Of these, $471$ are simplices and $162,\!479$ are simplicial. The simplex
\[
S_2^3=\sconv{(0,0,0),(2,0,0),(0,3,0),(0,0,18)}
\]
is the unique polytope maximising the volume. This same simplex also uniquely maximises both the boundary volume and number of lattice points, with
\[
\Vol{S_2^3}=108,\qquad\Vol{\bdry{S_2^3}}=102,\qquad\text{ and }\qquad\abs{S_2^3\cap\Z^3}=55.
\]
The maximum number of vertices, edges, and facets are, respectively, $18$, $30$, and $18$. After translating an interior lattice point to the origin, the largest dual volume is obtained by the $2$-point polytope
\[
P=\sconv{(-1,-1,-1),(0,-1,-1),(-1,0,-1),(7,7,8)},\qquad\text{ with }\qquad\Vol{P^*}=243/2.
\]
\end{thm}

In addition to being interesting in its own right, the classification of $2$-point three-dimensional polytopes is useful when studying Ehrhart theory. We explain why this is important in~\S\ref{subsec:dim3}; loosely speaking, the Ehrhart theory of hollow three-dimensional polytopes is well understood, however the case of $k$-point lattice polytopes ($k\geq 1$) remains open. In~\S\ref{sec:last} we formulate two new conjectural inequalities based on the classifications of $1$- and $2$-point polytopes (see Conjecture~\ref{conj:3d}). If true, these inequalities would prove a long-standing conjectural bound on the maximum volume of a $k$-point polytope in three dimensions (see Conjecture~\ref{conj:volume_bound}).

We briefly mention an important possible application of this classification to algebraic geometry. Let $P$ be a $2$-point polytope and, after possible translation, assume that the origin is one of the two interior lattice points. Assume further that the vertices $\V{P}$ of $P$ are primitive. Then $P$ corresponds to a three-dimensional projective toric variety $X_P$ whose fan is given by the cones spanning the faces of $P$. This variety is Fano and has log-terminal singularities. Recent progress in Mirror Symmetry for Fano varieties~\cite{CCGK16} suggests that many (not necessarily toric) terminal Fano threefolds $X$ will have $\Q$-Gorenstein degenerations to log-terminal toric Fano varieties $X_P$. The classification of $2$-point polytopes provides a source of possible toric $\Q$-Gorenstein degenerations. This opens the possibility of analysing the candidate terminal Fano threefolds of Alt{\i}nok--Brown--Reid~\cite{ABR02} via techniques from Mirror Symmetry: by matching the Hilbert series of a candidate terminal Fano threefold $X$ with that of $X_P$, where $P$ is a $2$-point polytope, one could attempt to exhibit the existence (or otherwise) of $X$.

\subsection{Ehrhart theory}
Let $P$ be a full-dimensional lattice polytope in $\Z^d$. The function $\ehr_P(k):=\abs{kP\cap\Z^d}$ counting the number of lattice points in the $k$-th dilation of $P$, $k\in\Z_{\geq 0}$, is given by a polynomial in $k$ of degree $d$ called the \emph{Ehrhart polynomial} of $P$~\cite{Ehr62}. An important open problem is to determine which degree $d$ polynomials correspond to Ehrhart polynomials. Associated with the Ehrhart polynomial is a rational function whose Taylor expansion gives a generating series for $\ehr_P$:
\[
\sum_{k\geq 0}\ehr_P(k)t^k=\frac{\h(t)}{(1-t)^{d+1}}.
\]
Here $\h(t)=\h_0+\h_1 t+\dots+\h_d t^d$ is a polynomial of degree at most $d$ with integer coefficients called the \emph{Ehrhart $\h$-polynomial} (or \emph{$\otherh$-polynomial}) of $P$. It is often convenient to identify the $\h$-polynomial with the vector of its coefficients $(\h_0,\h_1,\ldots,\h_d)$, which is called the \emph{Ehrhart $\h$-vector} (or \emph{$\otherh$-vector}) of $P$. We refer to~\cite{BR15} for additional background material.

\begin{question}\label{question}
For each $d$, is it possible to characterise those vectors $(\h_0,\h_1,\ldots,\h_d)$ which are $\h$-vectors for some $d$-dimensional lattice polytope? 
\end{question}

Although characterising the Ehrhart polynomials $L_P$ or the $\h$-vectors are equivalent problems, the $\h$-vectors have a better understood combinatorial interpretation. Ehrhart~\cite{Ehr62} showed that:
\begin{equation}\label{eq:ehr}
\h_0=1,\qquad
\h_1=\abs{P\cap\Z^d} - d - 1\geq\abs{\intr{P}\cap\Z^d}=\h_d,\qquad
\sum_{i=0}^d\h_i=\Vol{P}.
\end{equation}
Here $\Vol{P}=d!\vol{P}$ is the \emph{normalised volume} of the polytope $P$. If $S\subset\R^d$ is a $d$-dimensional simplex with vertices $\{v_0,\ldots,v_d\}$, then $\h_k$ counts the number of lattice points in the half-open parallelepiped
\[
\Pi(S):=\left\{\sum_{i=0}^d\lambda_i (1,v_i)\ \Big|\ 0\leq\lambda_i < 1\right\}\subset\R^{d+1}
\]
having first coordinate equal to $k$. This can be generalised to polytopes using the half-open triangulations approach of K{\"o}ppe--Verdoolaege~\cite{KV08}. As a consequence, the $\h$-vectors have non-negative entries (a result originally due to Stanley~\cite{Sta80} using techniques from commutative algebra).

In two dimensions the answer to Question~\ref{question} was first given by Scott~\cite{Sco76}:

\begin{thm}\label{thm:scott}
The vector with integer entries $(1,\h_1,\h_2)$ is the $\h$-vector of a two-dimensional lattice polygon if and only if one of the following conditions holds:
\begin{center}
\begin{enumerate*}[afterlabel={\ \ },itemjoin={\qquad or \qquad}]
\item $\h_2=0$;
\item $0 < \h_2\leq\h_1\leq 3\h_2+3$;
\item\label{case:scott_3} $(1,\h_1,\h_2)=(1,7,1)$.
\end{enumerate*}
\end{center}
In case~\eqref{case:scott_3} the polygon is unimodular equivalent to $3\Delta_2$, the third dilation of the standard simplex
\[
\Delta_2:=\sconv{(0,0),(1,0),(0,1)}.
\]
\end{thm}

Although the answer to Question~\ref{question} in higher dimensions remains open, several inequalities on the entries of the $\h$-vector are known~\cite{BM85,Hib90,Sta91,Sta09,Sta16}. Of particular relevance is a result due to Hibi~\cite{Hib94}: if $P$ is not hollow then
\begin{equation}\label{ineq:hibi}
1\leq\h_1\leq\h_i\qquad\text{ for }i=2,\ldots,d-1.
\end{equation}

In arbitrary dimension, one of the most interesting challenges is to bound the volume $\Vol{P}$ of a lattice polytope $P$ in terms of the number of interior points $\abs{\intr{P}\cap\Z^d}$, or equivalently, bound $\h_1+\cdots+\h_{d-1}$ in terms of $\h_d$. This is of course not possible if $P$ is hollow; for example, the hollow triangle $\sconv{(0,0),(m,0),(0,1)}$ has volume $m$ and so can be made arbitrarily large. In the case when $P$ is not hollow, the first general result was proven by Hensley~\cite{Hen83}, and later improved upon by Lagarias--Ziegler~\cite{LZ91} and Pikhurko~\cite{Pik01}:

\begin{thm}\label{thm:Pik}
Let $P$ be a $d$-dimensional $k$-point polytope, $k\geq 1$. Then:
\begin{equation}\label{eq:Pik}
\Vol{P}\leq d!\cdot(8d)^d\cdot 15^{d\cdot 2^{2d+1}}\cdot k
\end{equation}
\end{thm}

The bound in Theorem~\ref{thm:Pik} appears to be far from sharp. With this in mind, Zaks--Perles--Wilkes~\cite{ZPW82} defined the $d$-dimensional simplex
\begin{equation}\label{eq:ZPW_simplex}
S^d_k:=\sconv{\orig,s_1 e_1,\ldots,s_{d-1}e_{d-1},(k+1)(s_d-1)e_d},\qquad\text{ where }k\geq 1.
\end{equation}
Here $(s_i)_{i\in\Z_{\geq 1}}$ is the \emph{Sylvester sequence}
\[
s_1=2,\qquad s_i=s_1\cdots s_{i-1}+1.
\]
It is reasonable to conjecture that, for fixed $d$ and $k$, the simplex $S_k^d$ maximises the volume amongst all $k$-point $d$-dimensional polytopes. Hints of this conjecture can be tracked back to~\cite{ZPW82}, \cite{Hen83} and \cite{LZ91}.

\begin{conjecture}\label{conj:volume_bound}
Fix $d\geq 3$ and $k\geq 1$. A $k$-point $d$-dimensional lattice polytope $P$ satisfies
\begin{equation}\label{eq:conj}
\Vol{P}\leq (k+1)(s_d-1)^2.
\end{equation}
With the exception of the case when $d=3$, $k=1$, this inequality is an equality if and only if $P=S^d_k$.
\end{conjecture}

The two-dimensional case is not included in Conjecture~\ref{conj:volume_bound}; here the volume bound follows from Theorem~\ref{thm:scott} and requires a different formulation.  The case when $d=3$, $k=1$ has been addressed in~\cite{Kas10}: in addition to $S^3_1$, the maximum volume of $72$ is also attained by the simplex
\[
\sconv{(0,0,0),(2,0,0),(0,6,0),(0,0,6)}.
\]
Averkov--Kr\"{u}mpelmann--Nill~\cite{AKN15} prove a ``simplicial'' version of Conjecture~\ref{conj:volume_bound} when $k=1$: $S^d_1$ is the unique simplex with maximum volume among all $1$-point $d$-dimensional simplices, $d\geq 4$. A ``dual'' version of Conjecture~\ref{conj:volume_bound} when $k=1$ is proved in~\cite{BKN16} and, as a consequence, Conjecture~\ref{conj:volume_bound} is true when one considers only reflexive polytopes.

\subsection{Dimension three}\label{subsec:dim3}
It makes sense to consider Question~\ref{question} in the first unsolved case, the three-dimensional one. Treutlein~\cite[Theorem~2]{Tre10} generalises Theorem~\ref{thm:scott} to polytopes of degree two, i.e.\ to polytopes whose $\h$-polynomial has degree two. Such polytopes are necessarily hollow. Henk--Tagami show that this generalisation is sufficient~\cite[Proposition~1.10]{Hen09}, giving the following characterisation:

\begin{thm}
The vector with integer entries $(1,\h_1,\h_2)$ is the $\h$-vector of a hollow three-dimensional lattice polytope if and only if one of the following conditions holds:
\begin{center}
\begin{enumerate*}[afterlabel={\ \ },itemjoin={\qquad or \qquad}]
\item $\h_2=0$;
\item $0\leq\h_1\leq 3\h_2+3$;
\item $(1,\h_1,\h_2)=(1,7,1)$.
\end{enumerate*}
\end{center}
\end{thm}

The next natural step is to consider three-dimensional lattice polytopes with interior points, i.e\ such that $\abs{\intr{P}\cap\Z^3}=\h_3>0$. In this case Lagarias--Ziegler~\cite{LZ91} tell us that, for fixed $\h_3$, there are only finitely many pairs $(\h_1,\h_2)$ such that $(1,\h_1,\h_2,\h_3)$ is the $\h$-vector of a three-dimensional lattice polytope. The inequalities~\eqref{eq:ehr},~\eqref{ineq:hibi}, and~\eqref{eq:Pik} give:
\begin{equation}\label{eq:a}
\h_3\leq\h_1\leq\h_2
\end{equation}
\begin{equation}\label{eq:b}
\Vol{P}\leq 3!\cdot 24^3\cdot 15^{3\cdot 2^7}\cdot\h_3\approx 2^{1517}\cdot\h_3
\end{equation}
Although other bounds are known for the $\h$-vectors of special families of lattice polytopes, inequalities~\eqref{eq:a} and~\eqref{eq:b} summarise the current state of knowledge for the general case. Note that~\eqref{eq:a} is sharp: the polytope $\sconv{(-1,-1,-1),(1,0,0),(0,1,0),(0,0,m)}$ has $\h$-vector $(1,m,m,m)$. Inequality~\eqref{eq:b}, however, is presumed to be far from sharp: Conjecture~\ref{conj:volume_bound} gives the expected bound
\begin{equation}\label{eq:b2}
\Vol{P}\leq 36(\h_3+1).
\end{equation}

\subsection*{Organisation of the paper}
The classification of the $2$-points polytopes is achieved via a generalisation of the inductive algorithm~\cite{Kas10} used to classify the $1$-point polytopes. It is necessary to classify two families of lattice polytopes.
\begin{enumerate}
\item
In~\S\S\ref{sec:weights1}--\ref{sec:weights2} we classify the $1$- and $2$-point simplices of dimension at most three; the weights of these simplices will be used in the inductive step of the algorithm.
\item
In~\S\ref{sec:minimal} we describe all three-dimensional lattice polytopes satisfying a minimality condition (see Definition~\ref{def:minimal}); these minimal polygons form the base case of the algorithm.
\end{enumerate}
In~\S\ref{sec:main} we show the validity of the algorithm and describe the results of the classification, comparing them with other known classifications. In~\S\ref{sec:last} we analyse the data gathered and sketch the distribution of the $\h$-vectors of $1$- and $2$-point three-dimensional polytopes, comparing them with the inequalities of~\S\ref{subsec:dim3}

\section{Classification of one- and two-point triangles}\label{sec:weights1}
Let $S$ be a lattice simplex with $\V{S}=\{v_0,\ldots,v_d\}$ such that $\orig\in\intr{S}$. We say that $S$ has \emph{weights} $(\lambda_0,\ldots,\lambda_d)\in\Z^{d+1}_{>0}$ if
\[
\lambda_0v_0+\cdots+\lambda_dv_d=\orig.
\]
Since weights are unique up to scalar multiplication and reindexing, it is convenient to normalise them by requiring $\gcd{\lambda_0,\ldots,\lambda_d}=1$ and $\lambda_0\leq\cdots\leq\lambda_d$. Let $\cW_{d,k}$ denote the set of (normalised) weights for all $k$-point $d$-dimensional simplices. Trivially
\[
\cW_{1,1}=\{(1,1)\}\qquad\text{ and }\qquad\cW_{1,2}=\{(1,2)\}.
\]
The five $1$-point lattice triangles (depicted in Figure~\ref{fig:2d1ps}) give
\[
\cW_{2,1}=\{(1,1,1),(1,1,2),(1,2,3)\}.
\]
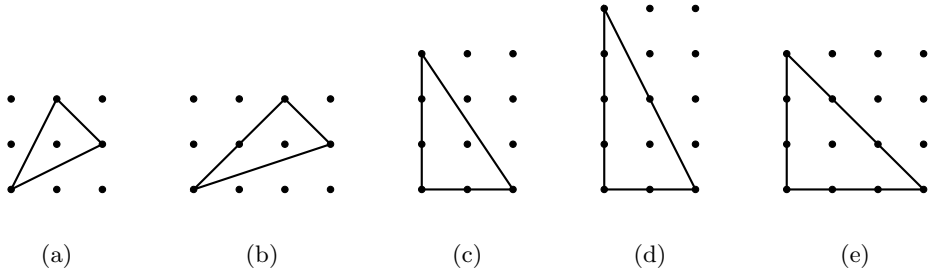
\begin{figure}[tb]
\centering
\begin{tikzpicture}[scale=0.6]
\draw[thick] (0,0)--(2,1)--(1,2)--cycle;
\draw[thick] (4,0)--(7,1)--(6,2)--cycle;
\draw[thick] (9,0)--(11,0)--(9,3)--cycle;
\draw[thick] (13,0)--(15,0)--(13,4)--cycle;
\draw[thick] (17,0)--(20,0)--(17,3)--cycle;
\foreach\x in {0,1,...,2}{
	\foreach\y in {0,1,...,2}{
		\draw[fill=black] (\x,\y) circle (0.2em) node[below]{};
	}
}
\foreach\x in {0,1,...,3}{
	\foreach\y in {0,1,...,2}{
		\draw[fill=black] (4+\x,\y) circle (0.2em) node[below]{};
	}
}
\foreach\x in {0,1,2}{
	\foreach\y in {0,1,...,3}{
		\draw[fill=black] (9+\x,\y) circle (0.2em) node[below]{};
	}
}
\foreach\x in {0,1,...,2}{
	\foreach\y in {0,1,...,4}{
		\draw[fill=black] (13+\x,\y) circle (0.2em) node[below]{};
	}
}
\foreach\x in {0,1,...,3}{
	\foreach\y in {0,1,...,3}{
		\draw[fill=black] (17+\x,\y) circle (0.2em) node[below]{};
	}
}
\draw[] (1,-1) circle (0em) node[below]{\hypertarget{subfig:P2}{\small(a)}};
\draw[] (5.5,-1) circle (0em) node[below]{\hypertarget{subfig:P112}{\small(b)}};
\draw[] (10,-1) circle (0em) node[below]{\small(c)};
\draw[] (14,-1) circle (0em) node[below]{\small(d)};
\draw[] (18.5,-1) circle (0em) node[below]{\small(e)};
\end{tikzpicture}
\caption{The five $1$-point lattice triangles. The corresponding weights are $(1,1,1)$, $(1,1,2)$, $(1,2,3)$, $(1,1,2)$, and $(1,1,1)$.}
\label{fig:2d1ps}
\end{figure}
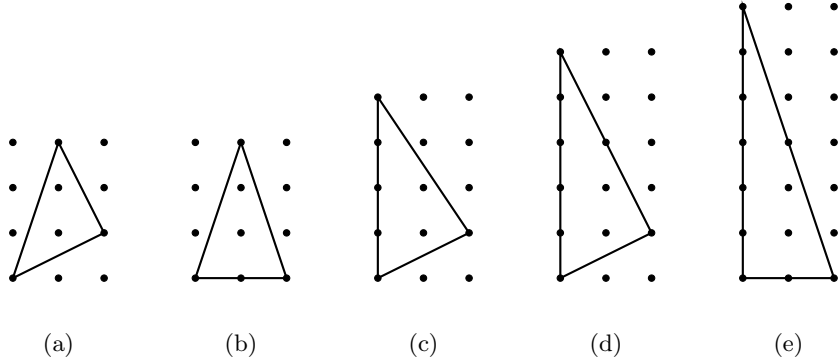
\begin{figure}[tb]
\centering
\begin{tikzpicture}[scale=0.6]
\draw[thick] (0,0)--(2,1)--(1,3)--cycle;
\draw[thick] (4,0)--(6,0)--(5,3)--cycle;
\draw[thick] (8,0)--(10,1)--(8,4)--cycle;
\draw[thick] (12,0)--(14,1)--(12,5)--cycle;
\draw[thick] (16,0)--(18,0)--(16,6)--cycle;
\foreach\x in {0,1,...,2}{
	\foreach\y in {0,1,...,3}{
		\draw[fill=black] (\x,\y) circle (0.2em) node[below]{};
	}
}
\foreach\x in {0,1,...,2}{
	\foreach\y in {0,1,...,3}{
		\draw[fill=black] (4+\x,\y) circle (0.2em) node[below]{};
	}
}
\foreach\x in {0,1,2}{
	\foreach\y in {0,1,...,4}{
		\draw[fill=black] (8+\x,\y) circle (0.2em) node[below]{};
	}
}
\foreach\x in {0,1,...,2}{
	\foreach\y in {0,1,...,5}{
		\draw[fill=black] (12+\x,\y) circle (0.2em) node[below]{};
	}
}
\foreach\x in {0,1,...,2}{
	\foreach\y in {0,1,...,6}{
		\draw[fill=black] (16+\x,\y) circle (0.2em) node[below]{};
	}
}
\draw[] (1,-1) circle (0em) node[below]{\small(a)};
\draw[] (5,-1) circle (0em) node[below]{\small(b)};
\draw[] (9,-1) circle (0em) node[below]{\small(c)};
\draw[] (13,-1) circle (0em) node[below]{\small(d)};
\draw[] (17,-1) circle (0em) node[below]{\small(e)};
\end{tikzpicture}
\caption{The five $2$-point lattice triangles, up to unimodular equivalence.}
\label{fig:2d2ps}
\end{figure}
\begin{lemma}\label{lem:2d2ps}
There are five $2$-point lattice triangles, up to unimodular equivalence. These are depicted in Figure~\ref{fig:2d2ps}. After translating each of the two interior points to the origin, there are eight possible weights:
\[
\cW_{2,2}=\{(1,1,1),(1,1,3),(1,2,2),(1,1,4),(1,2,3),(1,3,4),(1,4,5),(2,3,5)\}.
\]
\end{lemma}
\begin{proof}
Let $S=\sconv{v_1,v_2,v_3}$ be a triangle with two interior lattice points $\o1$ and $\o2$. The triangle $\Delta:=\sconv{\o1,\o2,v_1}$ is unimodular, so without loss of generality we may take $\o1=(0,0)$, $\o2=(1,0)$, and $v_1=(0,1)$. We rule out the possible lattice points for the remaining two vertices $v_2$ and $v_3$.

If $v_2$ is contained in any of the dark grey regions of Figure~\ref{fig:2triangles}, than it is not possible to choose any position for $v_3$ such that $S$ contains $\o1$ and $\o2$ in its strict interior. We therefore exclude those regions. Any remaining lattice points $u\notin\Delta$ define a cone of possible points $u'$ such that $\conv{\Delta\cup\{u'\}}$ contains $u$ in its interior. We can exclude all the points which are in the (strict) interior of any of these cones; in Figure~\ref{fig:2triangles} the excluded points are contained in the interior of the light grey regions. Since one of the vertices of $S$ must be below the line passing through $\o1$ and $\o2$, so there are only finitely many candidates for this vertex. As a consequence we obtain finitely many ways to choose the third vertex, and therefore $S$. After eliminating duplicates given by unimodular transformations we obtain five possible triangles with two interior lattice points. By choosing either $\o1$ or $\o2$ as the origin, we obtain the list $\cW_{2,2}$ of weights.
\end{proof}
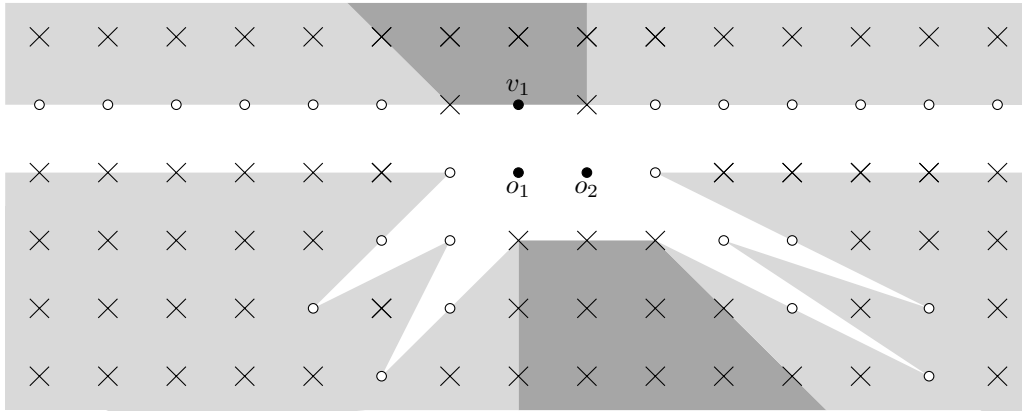
\begin{figure}[H]
\centering
\begin{tikzpicture}[scale=0.9]
\tikzset{cross/.style={cross out, draw=black, minimum size=2*(#1-\pgflinewidth), inner sep=0pt, outer sep=0pt},
cross/.default={0.5pt}}
\clip (-6.5,-1.5) rectangle (8.5,5.5); 
\fill[black!35] (1,1)--(3,1)--(5.5,-1.5)--(1,-1.5)--cycle;
\fill[black!15] (3,1)--(8,-1.5)--(5.5,-1.5)--cycle;
\fill[black!15] (1,1)--(-1.5,-1.5)--(1,-1.5)--cycle;
\fill[black!15] (2,3)--(2,4.5)--(3.5,4.5)--cycle;
\fill[black!15] (0,1)--(-1.25,-1.5)--(-6,-2)--cycle;
\fill[black!15] (3,2)--(8.5,2)--(8.5,-0.5)--(8.5,-0.5)--(9,-1)--cycle;
\fill[black!15] (4,1)--(10,-3)--(13,-2)--cycle;
\fill[black!15] (0,2)--(-6.5,2)--(-6.5,-0.5)--(-2.5,-0.5)--cycle;
\fill[black!15] (-6.5,-1.5)--(-2.5,-1.5)--(-2.5,1.5)--(-6.5,1.5)--cycle;
\fill[black!15] (0,3)--(-6.5,3)--(-6.5,4.5)--(-1.5,4.5)--cycle;
\fill[black!15] (2,3)--(8.5,3)--(8.5,4.5)--(2,4.5)--cycle;
\fill[black!35] (0,3)--(-1.5,4.5)--(2,4.5)--(2,3)--cycle;
\foreach\x in {-6,-5,...,-2,0,1,...,6,8}{\draw[fill=gray] (\x,-1) node[cross=4pt] {};}
\draw[fill=white] (-1,-1) circle (0.2em) node[below]{};
\draw[fill=white] (7,-1) circle (0.2em) node[below]{};
\foreach\x in {-6,-5,...,-3,-1,1,2,...,4,6,8}{\draw[fill=gray] (\x,0) node[cross=4pt] {};}
\draw[fill=white] (-2,0) circle (0.2em) node[below]{};
\draw[fill=white] (0,0) circle (0.2em) node[below]{};
\draw[fill=white] (5,0) circle (0.2em) node[below]{};
\draw[fill=white] (7,0) circle (0.2em) node[below]{};
\foreach\x in {-6,-5,...,-2,1,2,3,6,7,8}{\draw[fill=gray] (\x,1) node[cross=4pt] {};}
\draw[fill=white] (-1,1)		 circle 	(0.2em) node[below]{};
\draw[fill=white] (0,1)		 circle 	(0.2em) node[below]{};
\draw[fill=white] (4,1)		 circle 	(0.2em) node[below]{};
\draw[fill=white] (5,1)		 circle 	(0.2em) node[below]{};
\foreach\x in {-6,-5,...,-1,4,5,...,8}{\draw[fill=gray] (\x,2)		 node[cross=4pt] {};}
\foreach\x in {-6,-5,...,-1,3,4,...,8}{\draw[fill=white] (\x,3)		 circle 	(0.2em) node[below]{};}
\draw[fill=gray] (0,3)		 node[cross=4pt] {};
\draw[fill=gray] (2,3)		 node[cross=4pt] {};
\foreach\x in {-6,-5,...,8}{\draw[fill=gray] (\x,4)		 node[cross=4pt] {};}
\draw[fill=black] (1,2)		 circle 	(0.2em) node[below]{$\o1$};
\draw[fill=black] (2,2)		 circle 	(0.2em) node[below]{$\o2$};
\draw[fill=black] (1,3)		 circle 	(0.2em) node[above]{$v_1$};
\draw[fill=white] (3,2)		 circle 	(0.2em) node[below]{};
\draw[fill=white] (0,2)		 circle 	(0.2em) node[below]{};
\draw[fill=gray] (0,4)		 node[cross=4pt] {};
\draw[fill=gray] (1,4)		 node[cross=4pt] {};
\draw[fill=gray] (4,2)		 node[cross=4pt] {};
\draw[fill=gray] (5,2)		 node[cross=4pt] {};
\draw[fill=gray] (6,2)		 node[cross=4pt] {};
\draw[fill=gray] (7,2)		 node[cross=4pt] {};
\draw[fill=gray] (-1,0)		 node[cross=4pt] {};
\draw[fill=gray] (-1,2)		 node[cross=4pt] {};
\draw[fill=gray] (-1,4)		 node[cross=4pt] {};
\draw[fill=gray] (2,4)		 node[cross=4pt] {};
\draw[fill=gray] (3,4)		 node[cross=4pt] {};
\end{tikzpicture}
\caption{A proof of Lemma~\ref{lem:2d2ps}. The lattice points with a white circle are possible choices for $v_2$ and $v_3$.}
\label{fig:2triangles}
\end{figure}

\section{Classification of one- and two-point tetrahedra}\label{sec:weights2}
We now consider three-dimensional simplices. The weights of $\cW_{3,1}$ are already known, and are listed in~\cite[Table~3]{Kas10}. There are $104$ distinct weights, corresponding to $225$ $1$-point tetrahedra. To find $\cW_{3,2}$ we classify the $2$-point tetrahedra\footnote{In a personal communication, Christian Haase informs us that Noleen K\"ohler calculated an unpublished classification of the $2$-point tetrahedra. K\"ohler made use of conjectural bounds on the volume, and then iterated over all possible Hermite normal forms up to that volume.}. First we prove that each such tetrahedron $S$ decomposes in two (possibly lower-dimensional) $1$-point simplices; from this we construct all possible $S$.

\begin{lemma}\label{lem:subd}
Each $2$-point tetrahedron $S$, with $\intr{S}\cap\Z^3=\{\o1,\o2\}$, can be written as
\[
S=\conv{S_1\cup S_2},
\]
where $S_1$ and $S_2$ are two (possibly lower-dimensional) $1$-point lattice simplices satisfying:
\begin{enumerate}
\item
$\o2\in\V{S_1}$ and $\o1\in\V{S_2}$;
\item
$\intr{S_1}\cap\Z^3=\{\o1\}$ and $\intr{S_2}\cap\Z^3=\{\o2\}$;
\item
$\abs{\V{S_1}\cap\V{S_2}}=\dim{S_1} + \dim{S_2} - 4$.
\end{enumerate}
\end{lemma}
\begin{proof}
Consider the line passing through $\o1$ and $\o2$. This line intersects $\bdry{P}$ in two distinct points $q_1$ and $q_2$. Label these points so that $\o1$ lies on the line segment joining $q_1$ and $\o2$. Let $F_1$ and $F_2$ be the faces of $P$ of smallest possible dimension containing $q_1$ and $q_2$, respectively. Define $S_1=\conv{F_1\cup\{\o2\}}$ and $S_2=\conv{F_2\cup\{\o1\}}$. By construction $S_1$ and $S_2$ satisfy the hypothesis.
\end{proof}
\noindent
Fix a decomposition of $S$ with $\dim{S_1}\leq\dim {S_2}$. Then one of two possibilities holds:
\begin{center}
\begin{enumerate*}[label=({\Alph*}),ref={\Alph*},itemjoin={\qquad or \qquad},afterlabel={\ \ }]
\item\label{case:S2_dim3} $\dim{S_2}=3$; 
\item\label{case:S1_S2_dim2} $\dim{S_1}=\dim{S_2}=2$.
\end{enumerate*}
\end{center}
Note that $\dim{S_1}=1$ and $\dim{S_2}\leq 2$ is impossible; for example, $\conv{S_1\cup S_2}$ would have dimension at most two. We now describe how to systematically construct all $2$-point tetrahedra.

\subsection*{Possibility~\ref{case:S2_dim3}}\ \vspace{0.5em}\\
Begin by fixing the vertices $\{v_0,v_1,v_2,v_3\}$ of a $1$-point tetrahedron $S_2$ with unique interior point $\o2$, along with the dimension $1\leq k\leq 3$ and weights $(\lambda_0,\ldots,\lambda_k)\in\cW_{k,1}$ for a $1$-point simplex $S_1$. We will assume that $S_1$ and $S_2$ form a decomposition of a $2$-point tetrahedron $S=\conv{S_1\cup S_2}$ as in Lemma~\ref{lem:subd}, and either derive all possible vertices $\{u_0,\ldots,u_k\}$ for $S_1$ (and hence $S$), or deduce that the choice of weights $(\lambda_0,\ldots,\lambda_k)$ is incompatible with $S_2$.

Since $\dim{S_2}=3$, so $k-1$ vertices of $S_1$ are in common with $\V{S_2}$, with the remaining two vertices of $S_2$ given by $\o2\in\intr{S_2}$ and by $u_k\notin S_2$ (which we will determine). Furthermore, the unique interior lattice point $\o1\in\intr{S_1}$ is a vertex of $S_2$. This is illustrated in Figure~\ref{fig:S_is_determined}. Pick a subset $\{i_1,\ldots,i_{k-1}\}\subset\{0,1,2,3\}$, $i_1<\cdots<i_{k-1}$, an element $i_0\in\{0,1,2,3\}\setminus\{i_1,\ldots,i_{k-1}\}$, and set
\[
u_0=\o2,\quad u_1=v_{i_1},\quad\ldots,\quad u_{k-1}=v_{i_{k-1}},\quad\text{ and }\quad\o1=v_{i_0}.
\]
Let $\tau\in\Sym(k+1)$ be a permutation of the integers $\{0,\ldots,k\}$, and set
\[
u_k=\o1-\frac{1}{\lambda_{\tau(k)}}\sum_{j=0}^{k-1}\lambda_{\tau(j)}(u_j-\o1).
\]
If $u_k\in\Z^3$ we then check whether $S=\sconv{v_0,v_1,v_2,v_3,u_k}$ is a $2$-point tetrahedron. By iterating over all subsets $\{i_1,\ldots,i_{k-1}\}$, choices for $i_0$, and permutations $\tau\in\Sym(k+1)$, we will construct all possible $2$-point tetrahedra $S$ (not necessarily distinct with respect to unimodular equivalence) that can be obtained from $S_2$ and a $1$-point $k$-simplex $S_1$ with weights $(\lambda_0,\ldots,\lambda_k)$.
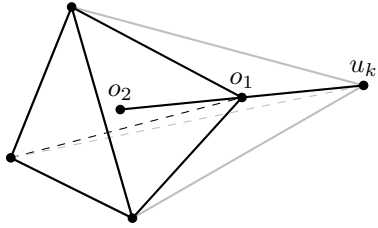
\begin{figure}[tb]
\centering
\begin{tikzpicture}[scale=0.8]
\draw[dashed] (0,0)--(3.8,1);
\draw[dashed,black!25] (0,0)--(5.8,1.2);
\draw[thick,black!25] (2,-1)--(0,0)--(1,2.5)--(2,-1)--(5.8,1.2)--(1,2.5);
\draw[thick] (2,-1)--(0,0)--(1,2.5)--(2,-1)--(3.8,1)--(1,2.5);
\draw[thick] (1.8,0.8)--(5.8,1.2);
\draw[fill=black] (1.8,0.8) circle (0.2em) node[above]{$\o2$};
\draw[fill=black] (3.8,1) circle (0.2em) node[above]{$\o1$};
\draw[fill=black] (5.8,1.2) circle (0.2em) node[above]{$u_k$};
\draw[fill=black] (2,-1) circle (0.2em) node[above]{};
\draw[fill=black] (0,0) circle (0.2em) node[above]{};
\draw[fill=black] (1,2.5) circle (0.2em) node[above]{};
\end{tikzpicture}
\caption{The position of the $k$-dimensional simplex $S_1$ (in this case, $k=1$) with respect to the tetrahedron $S_2$ is determined after choosing a subset of the vertices of $S_2$. The $2$-point tetrahedron $S=\conv{S_1\cup S_2}$ is determined (up to unimodular equivalence) by this choice.}
\label{fig:S_is_determined}
\end{figure}
\begin{remark}\label{rem:special_simplices}
The relative interiors of the faces of $S_1$ having $\o2$ as a vertex (similarly, those faces of $S_2$ having $\o1$ as a vertex) are in the (strict) interior of the $2$-point simplex $S$. Hence we can insist that $S_1$ has at least one vertex $v$ such that all the faces of $S_1$ containing $v$ have no lattice points in their relative interiors. Similarly for $S_2$. Out of the five $1$-point triangles only two satisfy this property, whilst out of the $225$ $1$-point tetrahedra only $63$ satisfy this property. This dramatically reduces the number of simplices needed to be considered.
\end{remark}

\subsection*{Possibility~\ref{case:S1_S2_dim2}}\ \vspace{0.5em}\\
We now handle the possibility that $\dim{S_1}=\dim{S_2}=2$. In this case, $S_1$ and $S_2$ have no vertices in common. Let $\Delta_{(1,1,1)}$ and $\Delta_{(1,1,2)}$ denote the triangles~(\hyperlink{subfig:P2}{a}) and~(\hyperlink{subfig:P112}{b}) of Figure~\ref{fig:2d1ps}, respectively. 

\begin{lemma}\label{lemma:edgeedge}
Let $S$ be a $2$-point tetrahedron decomposing in two $1$-point triangles. Then, up to unimodular equivalence, we have that $S_1$ is either
\[
\sconv{(1,0,0),(0,1,0),(-1,-1,0)}
\qquad\text{ or }\qquad
\sconv{(2,1,0),(0,1,0),(-1,-1,0)}.
\]
Moreover, if $v=(a,b,c)$ is one of the two vertices of $S_2$ distinct from $\o1$, we have that $0\leq a\leq b <c\leq 7$.
\end{lemma}
\begin{proof}
We use the notation introduced in the proof of Lemma~\ref{lem:subd}, and take $\o1=\orig$ to be the origin of the lattice. Let $v_1$ and $v_2$ be the endpoints of the edge $F_1$, and let $v_3$ and $v_4$ be the endpoints of the edge $F_2$, so that $S_1=\sconv{v_1,v_2,\o2}$ and $S_2=\sconv{v_3,v_4,\o1}$. As observed in Remark~\ref{rem:special_simplices}, $S_1$ and $S_2$ can have non-vertex boundary lattice points only on the edges $F_1$ and $F_2$, respectively. Since the only $1$-point lattice triangles not having lattice points in the relative interior of at least two edges are $\Delta_{(1,1,1)}$ and $\Delta_{(1,1,2)}$, this implies the first part of the statement. For the second part, we consider three cases:
\begin{center}
\begin{enumerate*}[label={(\alph*)},ref=\alph*,itemjoin=\qquad]
\item\label{edgeedge:a} $S_1=S_2=\Delta_{(1,1,1)}$;
\item\label{edgeedge:b} $S_1=\Delta_{(1,1,1)}$, $S_2=\Delta_{(1,1,2)}$;
\item\label{edgeedge:c} $S_1=S_2=\Delta_{(1,1,2)}$.
\end{enumerate*}
\end{center}

\noindent\textbf{Case~(\ref{edgeedge:a})}\hspace{0.5em} $S_1=S_2=\Delta_{(1,1,1)}$.\\
Suppose that $v_1=(1,0,0)$, $v_2=(0,1,0)$, and $\o2=(-1,-1,0)$. Let $a,b,c\in\Z$ be such that $v_3=(a,b,c)$. By applying a suitable unimodular transformation we can assume that
\begin{equation}\label{eq:cond}
0\leq a\leq b<c,
\end{equation}
where the second inequality can be assumed thanks to the symmetry of the first and second coordinates of the vertices of $S_1$. The choice of $v_3$ fixes $v_4$, and a trivial calculations shows that
\[
v_4=(-3-a,-3-b,-c).
\]
We now find conditions on $a,b,c$ by excluding situations which give rise to interior lattice points for $S$ distinct from $\o1$ and $\o2$. Notice that the lattice point $(1,1,1)$ can be written as
\[
\tfrac{1}{c}(a,b,c)+\tfrac{c-a}{c}(1,0,0)+\tfrac{c-b}{c}(0,1,0)+\tfrac{a+b-c-1}{c}(0,0,0).
\]
Since the numerators $1$, $c-a$, and $c-b$ are positive, we require $a+b-c-1\leq 0$, otherwise $(1,1,1)$ would be an interior point of $S$. Hence
\begin{equation}\label{eq:1}
c\geq a+b-1.
\end{equation}
Similarly the lattice point $(0,1,1)$, which can be written as
\[
\tfrac{1}{c}(a,b,c)+\tfrac{a}{c}(-1,-1,0)+\tfrac{c+a-b}{c}(0,1,0)+\tfrac{b-2a-1}{c}(0,0,0),
\]
is in $\intr{S}$ if and only if $b-2a-1\geq 0$. Therefore
\begin{equation}\label{eq:2}
b < 2a+1.
\end{equation}
Finally, since we can write $(-1,-1,-1)$ as
\[
\tfrac{1}{c}(-3-a,-3-b,-c)+\tfrac{c-a-3}{c}(-1,-1,0)+\tfrac{b-a}{c}(0,1,0)+\tfrac{2a-b+2}{c}(0,0,0),
\]
so this is a point in $\intr{P}$ if and only if $c-a-3\geq 0$. Since $2a-b+2>0$ by~\eqref{eq:2}, so we require that
\begin{equation}\label{eq:3}
c<a+3.
\end{equation}
From~\eqref{eq:cond},~\eqref{eq:1}, and~\eqref{eq:3}, we obtain the bounds $0\leq a\leq b <c\leq 5$.

\vspace{0.5em}
\noindent\textbf{Case~(\ref{edgeedge:b})}\hspace{0.5em} $S_1=\Delta_{(1,1,1)}$, $S_2=\Delta_{(1,1,2)}$.\\
As in the previous case we suppose that $v_1=(1,0,0)$, $v_2=(0,1,0)$, $\o2=(-1,-1,0)$, and $v_3=(a,b,c)$ for some integers $a$, $b$, and $c$ satisfying~\eqref{eq:cond}. In this case
\[
v_4=(-4-a,-4-b,-c).
\]
Note that~\eqref{eq:1} and~\eqref{eq:2} still hold. Moreover, since $(-1,-1,-1)$ can be written as
\[
\tfrac{1}{c}(-4-a,-4-b,-c)+\tfrac{c-a-4}{c}(-1,-1,0)+\tfrac{b-a}{c}(0,1,0)+\tfrac{2a-b+3}{c}(0,0,0),
\]
this is a point in $\intr{S}$ if and only if $c-a-4\geq 0$. By~\eqref{eq:2} we have that $2a-b+3>0$, therefore
\begin{equation}\label{eq:4}
c<a+4.
\end{equation}
From~\eqref{eq:cond},~\eqref{eq:1}, and~\eqref{eq:4}, we obtain the bounds $0\leq a\leq b <c\leq 7$.

\vspace{0.5em}
\noindent\textbf{Case~(\ref{edgeedge:c})}\hspace{0.5em} $S_1=S_2=\Delta_{(1,1,2)}$.\\
In this case, we suppose that $v_1=(2,1,0)$, $v_2=(0,1,0)$, $\o2=(-1,-1,0)$, and $v_3=(a,b,c)$ for some integers $a$, $b$ and $c$ satisfying
\begin{equation}\label{eq:cond2}
0\leq a,b<c.
\end{equation}
In this case $v_4=(-4-a,-4-b,-c)$. Suppose for a contradiction that $b<a$. The lattice point
\[
(0,1,1)=\tfrac{1}{c}(a,b,c)+\tfrac{a}{c}(-1,-1,0)+\tfrac{a-b}{c}(0,1,0)+\tfrac{c-2a+b-1}{c}(0,0,0)
\]
is in $\intr{S}$ if and only if $c-2a+b-1\geq 0$. Therefore
\begin{equation}
\label{eq:6}
c<2a-b+1.
\end{equation}
The point $(-1,0,-1)$ can be written as
\[
\tfrac{1}{c}(-4-a,-4-b,-c)+\tfrac{c-a-4}{c}(-1,-1,0)+\tfrac{c+b-a}{c}(0,1,0)+\tfrac{2a-b-c+3}{c}(0,0,0),
\]
and this is in $\intr{S}$ if and only if $2a-b-c+3\geq 0$. Hence $c>2a-b+3$, contradicting~\eqref{eq:6}. Thus $a\leq b$, and we see that~\eqref{eq:2} and~\eqref{eq:4} hold for this case. Furthermore, $(-1,-1,1)$ is equal to 
\[
\tfrac{1}{c}(-4-a,-4-b,-c)+\tfrac{a-c+4}{2c}(2,1,0)+\tfrac{2b-a-c+4}{2c}(0,1,0)+\tfrac{2c-b-5}{c}(0,0,0),
\]
and this is in $\intr{S}$ if and only if $a-c+4\geq 0$, $2b-a-c+4\geq 0$, and $2c-b-5>0$. Notice that $a-c+4>0$ by~\eqref{eq:4}, hence either $2b-a-c+4\geq0$ or $2c-b-5>0$ must fail to hold. If $2c-b-5\leq0$ we have $2c< c+5$, hence $c\leq4$. Conversely, suppose that $2b-a-c+4<0$. In this case the lattice point $(1,1,1)$ can be expressed as
\[
\tfrac{1}{c}(a,b,c)+\tfrac{c-a}{2c}(2,1,0)+\tfrac{a-2b+c}{2c}(0,1,0)+\tfrac{b-1}{c}(0,0,0).
\]
By hypothesis, $a-2b+c>0$, so $b-1$ has to be non-positive in order for $(1,1,1)\notin\intr{S}$. By~\eqref{eq:4} we obtain $c\leq 4$. Hence we have the bounds $0\leq a\leq b < c\leq 4$.
\end{proof}

Possibility~\ref{case:S2_dim3} gives rise to $460$ $2$-point tetrahedra (distinct up to unimodular equivalence). Possibility~\ref{case:S1_S2_dim2} gives rise to an additional $11$ $2$-point tetrahedra. We obtain:

\begin{thm}\label{thm:3s2p}
There are $471$ $2$-point tetrahedra, up to unimodular equivalence. After translating each of the two interior points to the origin, there are $548$ possible weights. These weights are listed in Table~\ref{tab:W32} on page~\pageref{tab:W32}.
\end{thm}

\section{Minimal polytopes in dimension three}\label{sec:minimal}
In this section we consider those $1$- or $2$-point three-dimensional lattice polytopes which are minimal under inclusion. These polytopes form the ``seeds'' from which the classification of $2$-point polytopes will be ``grown'' in~\S\ref{sec:main}.

\begin{definition}\label{def:minimal}
We say that a $d$-dimensional lattice polytope $P$ is \emph{minimal} if $P$ is not hollow and, for each vertex $v$ of $P$, the polytope $\conv{P\cap\Z^d\setminus\{v\}}$ is hollow.
\end{definition}

\begin{prop}\label{prop:small}
Let $P$ be a minimal three-dimensional polytope. Then $P$ is either a minimal $1$-point polytope or a non-hollow tetrahedron whose interior points lie on a line.
\end{prop}
\begin{proof}
We will show that each non-hollow three-dimensional polytope $P$ contains either a minimal $1$-point polytope, or a non-hollow tetrahedron whose interior points lie on a line. Since the statement is trivial for $1$-point polytopes, we suppose that $\abs{\intr{P}\cap\Z^3}\geq 2$. Let $\o1,\o2\in\intr{P}\cap\Z^3$ be two interior lattice points such that the line segment with end-points $\o1$ and $\o2$ contains no interior lattice points. Let $\ell$ be the line passing through $\o1$ and $\o2$, and let $q_1,q_2\in\ell\cap\bdry{P}$ be the points of intersection of $\ell$ with the boundary of $P$. We label $q_1$ and $q_2$ such that $\o1$ lies between $q_1$ and $\o2$.

Project $P$ via a map $\pi$ onto a plane orthogonal to $\ell$. Then $\pi(P)$ is a polygon in the lattice $\pi(\Z^3)\cong\Z^2$ with $o:=\pi(\o1)=\pi(\o2)\in\intr{\pi(P)}$. Progressively eliminate vertices of $\pi(P)$, moving to the convex hull of the remaining lattice points. Eventually we obtain a minimal polygon $Q$ containing $o$ as the unique interior lattice point. This minimal polygon $Q$ is either a triangle or a quadrilateral whose diagonals intersect in $o$ (see~\cite[Corollary 2.3]{Kas10}). Denote these two possibilities by $\triangle$ (the case when $Q$ is a triangle) and $\square$ (the case when $Q$ is a quadrilateral). Before considering these two possibilities, we fix some notation: if $v$ and $w$ are two points in $\R^3$, we denote by $[v,w]$, $[v,w)$, $(v,w]$, and $(v,w)$ the intervals $\sconv{v,w}$, $\sconv{v,w}\setminus\{w\}$, $\sconv{v,w}\setminus\{v\}$, and $\sconv{v,w}\setminus\{v,w\}$, respectively.

\begin{enumerate}
\item[$\triangle$]
If $Q$ is a triangle, there exist (at least) three lattice points of $P$ such that their image under $\pi$ are the vertices of $Q$. Let $T$ be the the lattice triangle in $\Z^3$ given by taking the convex hull of these three points. By construction $T$ intersects $\ell$ at one point $p\in\intr{T}$. Up to exchanging the roles of $\o1$ and $\o2$, there are two possibilities:
\begin{enumerate}[label={$\triangle$.\arabic*}]
\item If $p\in [q_1,\o1)$ then $\conv{T\cup\{\o2\}}$ is a tetrahedron with at least one interior lattice point.
\item If $p\in [\o1,\o2)$ then there are two sub-cases:
\begin{enumerate}[label*=.\arabic*]
\item If $(\o2,q_2)\cap\Z^3=\emptyset$, let $F$ be the smallest-dimensional face of $P$ containing $q_2$. Then $\conv{T\cup F}$ is a $1$-point polytope.
\item If $\o3\in(\o2,q_2)\cap\Z^3$ then $\conv{T\cup\{\o3\}}$ is a tetrahedron with at least one interior lattice point.
\end{enumerate}
\end{enumerate}
\item[$\square$]
If $Q$ is a quadrilateral, each of the two pairs of non-adjacent vertices is collinear with $o$. By pulling back the vertices of $Q$ via $\pi$, we find four vertices of $P$ that can be split into two pairs, each pair coplanar with $\ell$. Let $E_1$ and $E_2$ be the edges connecting each of the pairs of vertices, and let $p_1$ and $p_2$ the respective intersection with $\ell$. Without loss of generality there are three distinct configurations to consider:
\begin{enumerate}[label={$\square$.\arabic*}]
\item If $p_1\in [q_1,\o1)$ and $p_2\in [q_1,\o1]$ then either $\conv{E_1\cup E_2\cup\{\o2\}}$ is a $1$-point polytope, or $\conv{E_1\cup E_2}$ is a tetrahedron with at least one interior lattice point.
\item If $p_1\in [\o1,\o2)$ and $p_2\in [\o1,\o2]$ then there are two sub-cases:
\begin{enumerate}[label*=.\arabic*]
\item If $(\o2,q_2)\cap\Z^3=\emptyset$, let $F$ be the smallest-dimensional face of $P$ containing $q_2$. Then $\conv{E_1\cup E_2\cup F}$ is a $1$-point polytope.
\item If $(\o2,q_2)\cap\Z^3$ is not empty, choose $\o3\in(\o2,q_2)\cap\Z^3$ to be the closest lattice point to $\o2$. Then $\conv{E_1\cup E_2\cup \{\o3\}}$ is a $1$-point polytope.
\end{enumerate}
\item If $p_1\in [q_1,\o1)$ and $p_2\in (\o1,q_2]$ then $\conv{E_1\cup E_2}$ is a tetrahedron with at least one interior lattice point.\qedhere
\end{enumerate}
\end{enumerate}
\end{proof}

A consequence of Proposition~\ref{prop:small} is that a three-dimensional $2$-point polytope will always contain either a three-dimensional $1$-point polytope or a $2$-point tetrahedron. The minimal $1$-point polytopes were classified in~\cite{Kas10}. The minimal $2$-point tetrahedra are listed in Table~\ref{tab:minimal_2_points}.
\begin{table}[tb]
\centering
\begin{tabular}{ll}
\toprule
\multicolumn{1}{c}{Vertices}&\multicolumn{1}{c}{Interior Points}\\
\cmidrule(lr){1-1} \cmidrule(lr){2-2}
$(0,0,0),(1,0,0),(0,1,0),(5,5,8)$&$(1,1,1),(2,2,3)$\\
$(0,0,0),(1,0,0),(0,1,0),(6,9,10)$&$(1,1,1),(2,3,3)$\\
$(0,0,0),(1,0,0),(0,1,0),(5,7,24)$&$(1,1,3),(2,3,9)$\\
$(0,0,0),(1,0,0),(0,2,0),(5,2,6)$&$(1,1,1),(2,1,2)$\\
$(0,0,0),(1,0,0),(0,2,0),(5,2,12)$&$(1,1,2),(2,1,4)$\\
\bottomrule
\end{tabular}
\vspace{0.5em}
\caption{The five minimal $2$-point tetrahedra, up to unimodular equivalence.}
\label{tab:minimal_2_points}
\end{table}
\section{Classification of two-point polytopes in dimension three}\label{sec:main}
In this section we describe the algorithm used to generate the classification of three-dimensional $2$-point polytopes. The algorithm is essentially inductive, starting with a ``seed'' of minimal $1$-point polytopes and minimal $2$-point tetrahedra. The inductive step corresponds to ``growing'' the known polytopes by successive addition of vertices. We show that, subject to the requirement that there are no more than two interior lattice points, there are only finitely many ways to grow a polytope. Of course, the fact that there are a finite number of unimodular equivalence-classes of polytopes with a fixed non-zero number of interior lattice points is well-known~\cite{LZ91}. Nevertheless, in the proof of Proposition~\ref{prop:grow} below we describe an algorithmic approach to obtain a complete list of such classes.

\begin{prop}\label{prop:grow}
Let $K$ be a positive integer, and $P$ be $d$-dimensional polytope in $\Z^d$ with $ 1\leq\abs{\intr{P}\cap\Z^d}\leq K$. Then the set
\[
\cS:=\left\{v\in\Z^d\ \Big|\ \abs{\intr{\conv{P\cup\{v\}}}\cap\Z^d}\leq K\right\}
\]
is finite. 
\end{prop}
\begin{proof}
We prove that $\cS$ is a subset of a finite set $\cR$. Let $o\in\intr{P}\cap\Z^d$ be one of the interior lattice points of $P$; without loss of generality we may assume that $o=\orig$ is equal to the origin of the lattice. Fix a triangulation of $\bdry{P}$ obtained without introducing new vertices. Let $F$ be any $(n-1)$-dimensional non-empty face of the triangulation, and let $\{v_1,\ldots,v_n\}$ denote the set of vertices of $F$. Let $\ulambda=(\lambda_0,\ldots,\lambda_n)$ be a weight in $\bigcup_{k=1}^K\cW_{n,k}$, and let $\tau\in\Sym(n+1)$ be a permutation of the integers $\{0,\ldots,n\}$. Define
\[
v_{F,o,\ulambda,\tau}:=-\frac{1}{\lambda_{\tau(0)}}\sum_{i=1}^n\lambda_{\tau(i)}v_i,
\]
and define $\cR$ to be the set of all points $v_{F,o,\ulambda,\tau}$ such that $v_{F,o,\ulambda,\tau}\in\Z^d$, for all possible choices of $\tau$, $\ulambda$, $o$, and $F$. Since there are only finitely many such choices, it follows that $\cR$ is finite as well.

Let $v\in\cS$. Excluding a finite number of lattice points, we may suppose that $v\notin P$. Let $o\in\intr{P}\cap\Z^d$, and let $\ell$ be the line passing through $o$ and $v$. The line $\ell$ intersects $\bdry{P}$ in two distinct points. Let $v'\in\ell\cap\bdry{P}$ be the point of intersection furthest from $v$, and let $F$ be the smallest face of the triangulation of $\bdry{P}$ containing $v'$. Then $S:=\conv{F\cup\{v\}}$ is a $k$-point lattice simplex of dimension $n\leq d$, where $1\leq k\leq K$. By translating $o$ to the origin, $S$ has weights in $\cW_{n,k}$. Hence $v\in\cR$.
\end{proof}

\begin{algorithm}[tb]
\DontPrintSemicolon 
\KwData{The set $\cP$ of all minimal three-dimensional $1$-point polytopes and all minimal $2$-point tetrahedra. The sets $\cW_{d,k}$, for $k=1,2$ and $d=1,2,3$, of possible weights for the $k$-point $d$-dimensional simplices.}
\KwResult{The set $\cB$ of three-dimensional $1$- and $2$-point polytopes.}
$\cA\leftarrow\cP$\;
$\cB\leftarrow\emptyset$\;
\While{$\cA\neq\emptyset$}{
$P\leftarrow\mathrm{RandomElement}(\cA)$\;
$\Delta\leftarrow\mathrm{RandomTriangulation}(\bdry{P})$\;
\For{$F\in\Delta$}{
$n\leftarrow\dim{F}+1$\;
$\{v_1,\ldots,v_n\}\leftarrow\V{F}$\;
\For{$\ulambda\in\cW_{n,1}\cup\cW_{n,2}$}{
\For{$\tau\in\Sym(n+1)$}{
$v\leftarrow-\frac{1}{\lambda_{\tau(0)}}\sum_{i=1}^n\lambda_{\tau(i)}v_i$\;
\If{$v\in\Z^3$}{
$Q\leftarrow\conv{P\cup\{v\}}$\;
\If{$\abs{\intr{Q}\cap\Z^3}\leq 2$}{
\If{$Q\notin\cB$ \emph{(up to unimodular equivalence)}}{
$\cA\leftarrow\cA\cup\{Q\}$
}
}
}
}
}
}
$\cA\leftarrow\cA\setminus\{P\}$\;
$\cB\leftarrow\cB\cup\{P\}$\;
}
\caption{Classifying the three-dimensional $1$- and $2$-point polytopes.}
\label{alg:classify}
\end{algorithm}

The proof of Proposition~\ref{prop:small} is constructive, and allows us to classify all three-dimensional $2$-point polytopes. This is described in Algorithm~\ref{alg:classify}, which we implemented in the computer algebra system \textsc{Magma}~\cite{BCP97}. Source code and output can be downloaded from~\cite{GRDb}. The resulting classification is summarised in Theorem~\ref{thm:main}. An immediate consequence of this classification is:

\begin{cor}
Conjecture~\ref{conj:volume_bound} holds when $d=3$ and $k=2$.
\end{cor}

The results can be compared against existing classifications. Algorithm~\ref{alg:classify} produces an independent check of the classification~\cite{Kas10} of $1$-point polytopes. The recent classification by Blanco--Santos~\cite{BS16_1,BS16_2,BS16_3} of those three-dimensional polytopes having $\abs{P\cap\Z^3}\leq 11$ overlaps in part with the classification of $2$-point polytopes; as stated in~\cite{BS16_3}, where these two classifications intersect, they coincide. Table~\ref{table:invariants} summarises the maximum values of some common invariants; Tables~\ref{table:invariants2} and~\ref{table:invariants3} sketch the distribution of $2$-point polytopes with respect to volume and number of vertices.
\begin{table}[H]
\centering
\begin{tabular}{lcc}
\toprule
\multicolumn{1}{c}{Invariant}&Maximum&Polytopes attaining maximum\\
\cmidrule(lr){1-1} \cmidrule(lr){2-2} \cmidrule(lr){3-3}
Normalised volume&$108$ &$S^3_2$\\
Normalised boundary volume&$102$&$S^3_2$\\
Number of lattice points&$55$&$S^3_2$\\
Number of vertices&$18$&$2$ polytopes\\
Number of edges&$30$&$2$ polytopes\\
Number of facets&$18$&$31$ polytopes\\
\bottomrule
\end{tabular}
\vspace{0.5em}
\caption{A summary of the maximum values amongst three-dimensional $2$-point polytopes for some common invariants.}
\label{table:invariants}
\end{table}
\begin{table}[H]
\centering
\begin{tikzpicture}
\begin{axis}[
	ybar interval, 
	ymax=900000,
	ymin=0,
	xmin=7,
	xmax=108,
	height=4cm,
	width=15cm,
	xtick={10,20,30,40,50,60,70,80,90,100,108},
	xtickmin=7,
	xtickmax=108,
	xticklabel style={xshift=-0.67cm},
	xlabel={},
	ylabel={},
	xlabel style={below right},
	ylabel style={above left}
]
\addplot coordinates {
(1,0)(2,0)(3,0)(4,0)(5,0)(6,0)(7,1)(8,2)(9,7)(10,14)(11,14)(12,82)(13,61)(14,209)(15,245)(16,585)(17,688)(18,1456)(19,1845)(20,3439)(21,4524)(22,7505)(23,9686)(24,15604)(25,19687)(26,29028)(27,36680)(28,51499)(29,62690)(30,85188)(31,100925)(32,131652)(33,153639)(34,192614)(35,219376)(36,269040)(37,299745)(38,355632)(39,390438)(40,452757)(41,486021)(42,552091)(43,582397)(44,645820)(45,671365)(46,728653)(47,742880)(48,793785)(49,795889)(50,831158)(51,822694)(52,845345)(53,820191)(54,829751)(55,793881)(56,788445)(57,741347)(58,724692)(59,670737)(60,647376)(61,589164)(62,558020)(63,500223)(64,467229)(65,411371)(66,379315)(67,330225)(68,301490)(69,258146)(70,232024)(71,195079)(72,174531)(73,145701)(74,129739)(75,106885)(76,94982)(77,77127)(78,67657)(79,53720)(80,45995)(81,35399)(82,29338)(83,21524)(84,17185)(85,11948)(86,9071)(87,5968)(88,4430)(89,2757)(90,2008)(91,1215)(92,915)(93,565)(94,435)(95,269)(96,231)(97,144)(98,116)(99,70)(100,60)(101,36)(102,29)(103,15)(104,12)(105,5)(106,4)(107,1)(108,1)
};
\end{axis}
\end{tikzpicture}
\caption{Distribution of three-dimensional $2$-point polytopes by (normalised) volume. The most common volume is $52$, attained by $845,\!345$ polytopes.}
\label{table:invariants2}
\end{table}
\begin{table}[H]
\centering
\begin{tikzpicture}
\begin{axis}[
	ybar interval, 
	ymax=7000000,
	ymin=0,
	xmin=4,
	xmax=19,
	height=4cm,
	width=15cm,
	xtick={4,5,6,7,8,9,10,11,12,13,14,15,16,17,18,19},
	xtickmin=4,
	xtickmax=19,
	xlabel={},
	ylabel={},
	xlabel style={},
	ylabel style={left}
]
\addplot coordinates {
(4,471)(5,12925)(6,145854)(7,819712)(8,2614254)(9,5052089)(10,6137658)(11,4736026)(12,2314410)(13,699524)(14,126763)(15,12925)(16,810)(17,26)(18,2)
};
\end{axis}
\end{tikzpicture}
\caption{Distribution of three-dimensional $2$-point polytopes by number of vertices. The most common number of vertices is $10$, attained by $6,\!137,\!658$ polytopes.}
\label{table:invariants3}
\end{table}

\section{Distribution of the \protect{$\h$}-vectors of three-dimensional polytopes}\label{sec:last}
From the classifications of $1$- and $2$-point polytopes we can extract the possible $\h$-vectors, and can compare their distribution with the known inequalities. In Figure~\ref{fig:distribution} we plot all occurring pairs of coefficients $(\h_1,\h_2)$ of $\h$-vectors for one and two interior lattice points. Inequalities~\eqref{eq:a} and~\eqref{eq:b} of~\S\ref{subsec:dim3} define a bounded region of the plane in which such pairs can appear. We also plot inequalities~\eqref{eq:a} and~\eqref{eq:b2}, i.e.\ $\h_3\leq\h_1\leq\h_2$ and $\Vol{P}\leq 36(\h_3+1)$.
\begin{figure}[H]
\centering
\hfill
\subfigure[Polytopes with one interior point]{
\begin{tikzpicture}[trim axis left, trim axis right]
\begin{axis}[xmin=0,xmax=42,ymin=0,ymax=75,xlabel=$\h_1$,ylabel=$\h_2$,width=14cm,axis lines=middle,axis equal image,grid=both]

\addplot[black,mark=none,domain=0.5:38, thick]
{x} node[below,pos=0.5,yshift=22pt]{};



\addplot[black,dashed,mark=none,domain=0:38, thick]
{-x+70} node[below,pos=0.87]{};

\draw[black, thick] (axis cs:1,0) -- (axis cs:1,71) node [below,pos=0.94,xshift=9]{};
\pgfplotstableread{3_1.txt}\mydata;
\addplot [only marks,mark=*,mark options={scale=0.8,fill=white}] table [] {\mydata};

\pgfplotstableread{3s_1.txt}\mydata;
\addplot [color=black,only marks,mark=*,mark options={scale=0.8}] table [] {\mydata};
\end{axis}
\end{tikzpicture}
}
\hfill
\subfigure[Polytopes with two interior points]{
\begin{tikzpicture}[trim axis left, trim axis right]
\begin{axis}[xmin=0,xmax=65,ymin=0,ymax=111,xlabel=$\h_1$,ylabel=$\h_2$,width=14cm,axis lines=middle,axis equal image,grid=both]

\addplot[black,mark=none,domain=0:58, thick]
{x} node[below,pos=0.5,yshift=15pt]{};



\addplot[black,dashed,mark=none,domain=0:58, thick]
{-x+105} node[below,pos=0.87]{};


\draw[black, thick] (axis cs:2,0) -- (axis cs:2,105) node [below,pos=0.94,xshift=9]{};
\pgfplotstableread{3_2.txt}\mydata;
\addplot [only marks,mark=*,mark options={scale=0.5,fill=white}] table [] {\mydata};
\pgfplotstableread{3s_2.txt}\mydata;
\addplot [color=black,only marks,mark=*,mark options={scale=0.5}] table [] {\mydata};
\end{axis}
\end{tikzpicture}
}
\hfill\phantom{.}
\caption{The coefficients $(\h_1,\h_2)$ of the three-dimensional $1$- and $2$-point polytopes. The pairs $(\h_1,\h_2)$ attained by at least one simplex are denoted by $\bullet$. Inequalities~\eqref{eq:a} are marked with black lines, inequality~\eqref{eq:b2} with a dashed black line.}
\label{fig:distribution}
\end{figure}
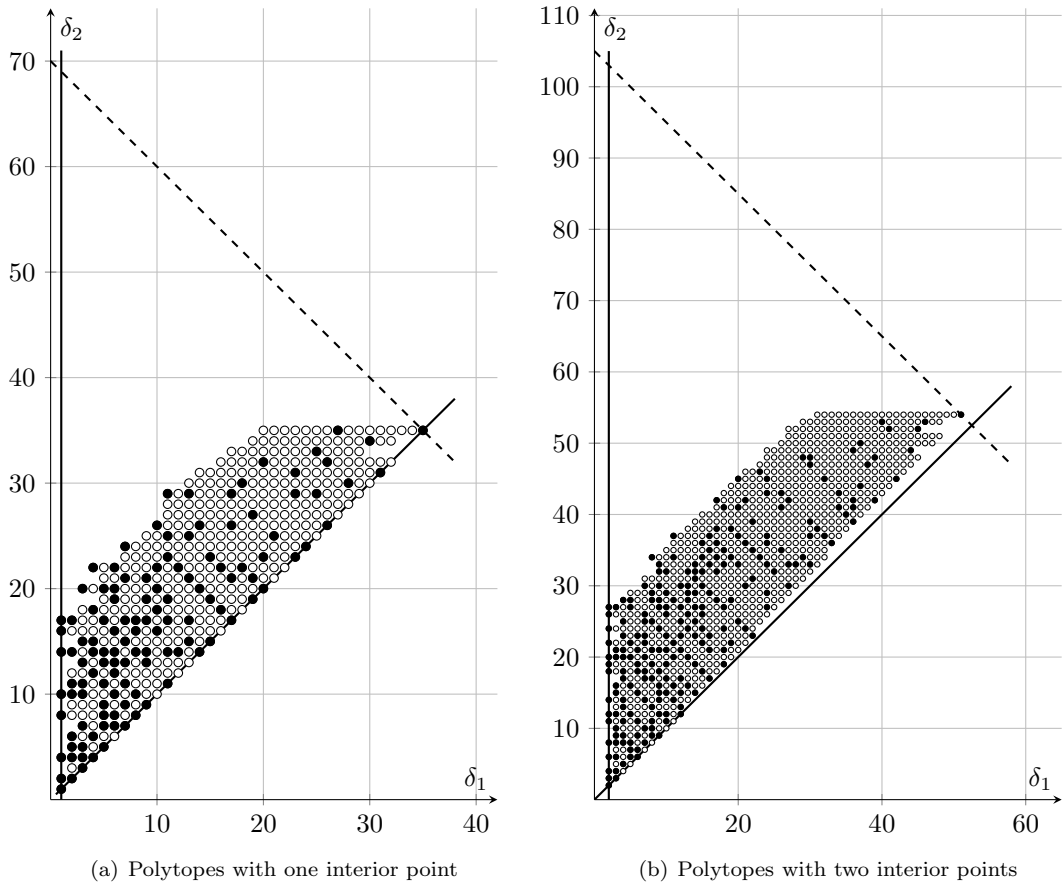
Even assuming the conjectured volume inequality~\eqref{eq:b2} to be true, all the pairs $(\h_1,\h_2)$ appear in a much smaller region than the one delimited by~\eqref{eq:a} and~\eqref{eq:b2}. In particular, the $\h$-vector of the maximal simplices $S^3_1$ and $S^3_2$, given by $(1,35,35,1)$ and $(1,51,54,2)$ respectively, acts as a bound for the other $\h$-coefficients.

For all $k\geq 1$, the number lattice points of the Zaks--Perles--Wilkes simplex $S^3_k$ has been calculated in~\cite{ZPW82}, and we know that $\abs{S^3_k\cap\Z^3}=16k+23$. Since the volume and the number of interior points of $S^3_k$ are also known, we can compute their $\h$-vector: this is equal to $(1,16k+19,19k+16,k)$.

\begin{conjecture}\label{conj:3d}
For any $\h$-vector $(1,\h_1,\h_2,\h_3)$ of a non-hollow three-dimensional polytope $P$, the inequalities
\[
\h_1\leq 16\h_3+19\qquad\text{ and }\qquad\h_2\leq 19\h_3+16
\]
hold. Moreover, the first inequality is an equality if and only if $P=S^3_{\h_3}$.
\end{conjecture}
\noindent
Note that Conjecture~\ref{conj:3d} implies Conjecture~\ref{conj:volume_bound} in dimension three.

\subsection*{Acknowledgments}
GB is supported by the Stiftelsen GS Magnusons Fund, by a Jubileumsfond grant from the Knut and Alice Wallenbergs Foundation, and by Vetenskapsr{\aa}det grant~NT:2014-3991. AK is supported by EPSRC Fellowship~EP/N022513/1. The majority of this work was done during a visit of GB to AK, funded by the University of Nottingham's EPSRC Institutional Sponsorship~EP/N508822/1.
\begin{center}
\small
\begin{longtable}{c@{\ \ }cc@{\ \ }cc@{\ \ }cc@{\ \ }cc@{\ \ }c}
\caption{The $548$ distinct weights $(\lambda_0,\lambda_1,\lambda_2,\lambda_3)\in\cW_{3,2}$ occurring for the $471$ $2$-point tetrahedra. The weights are listed in order by the sum $h=\lambda_0+\lambda_1+\lambda_2+\lambda_3$.}
\label{tab:W32}\\
\toprule
Weights&$h$&Weights&$h$&Weights&$h$&Weights&$h$&Weights&$h$\\
\cmidrule(lr){1-2} \cmidrule(lr){3-4} \cmidrule(lr){5-6} \cmidrule(lr){7-8} \cmidrule(lr){9-10}
\endfirsthead
\multicolumn{10}{l}{\vspace{-0.7em}\tiny Continued from previous page.}\\
\addlinespace[1.7ex]
\midrule
Weights&$h$&Weights&$h$&Weights&$h$&Weights&$h$&Weights&$h$\\
\cmidrule(lr){1-2} \cmidrule(lr){3-4} \cmidrule(lr){5-6} \cmidrule(lr){7-8} \cmidrule(lr){9-10}
\endhead
\multicolumn{10}{r}{\raisebox{0.2em}{\tiny Continued on next page.}}\\
\endfoot
\bottomrule
\endlastfoot
\oddrow \oddrow $(1,1,1,1)$&$4$&$(1,1,1,2)$&$5$&$(1,1,1,3)$&$6$&$(1,1,2,2)$&$6$&$(1,1,1,4)$&$7$\\
\evnrow $(1,1,2,3)$&$7$&$(1,2,2,2)$&$7$&$(1,1,1,5)$&$8$&$(1,1,2,4)$&$8$&$(1,1,3,3)$&$8$\\
\oddrow $(1,2,2,3)$&$8$&$(1,1,1,6)$&$9$&$(1,1,2,5)$&$9$&$(1,1,3,4)$&$9$&$(1,2,2,4)$&$9$\\
\evnrow $(1,2,3,3)$&$9$&$(2,2,2,3)$&$9$&$(1,1,2,6)$&$10$&$(1,1,3,5)$&$10$&$(1,1,4,4)$&$10$\\
\oddrow $(1,2,2,5)$&$10$&$(1,2,3,4)$&$10$&$(2,2,3,3)$&$10$&$(1,1,2,7)$&$11$&$(1,1,4,5)$&$11$\\
\evnrow $(1,2,3,5)$&$11$&$(1,3,3,4)$&$11$&$(2,2,3,4)$&$11$&$(1,1,2,8)$&$12$&$(1,1,3,7)$&$12$\\
\oddrow $(1,1,4,6)$&$12$&$(1,2,3,6)$&$12$&$(1,2,4,5)$&$12$&$(1,3,3,5)$&$12$&$(1,3,4,4)$&$12$\\
\evnrow $(2,2,3,5)$&$12$&$(2,3,3,4)$&$12$&$(1,1,5,6)$&$13$&$(1,2,3,7)$&$13$&$(1,2,4,6)$&$13$\\
\oddrow $(1,3,4,5)$&$13$&$(2,3,3,5)$&$13$&$(1,1,5,7)$&$14$&$(1,2,3,8)$&$14$&$(1,2,4,7)$&$14$\\
\evnrow $(1,2,5,6)$&$14$&$(1,3,3,7)$&$14$&$(1,3,4,6)$&$14$&$(1,4,4,5)$&$14$&$(2,2,3,7)$&$14$\\
\oddrow $(2,3,4,5)$&$14$&$(1,1,3,10)$&$15$&$(1,1,6,7)$&$15$&$(1,2,3,9)$&$15$&$(1,2,5,7)$&$15$\\
\evnrow $(1,3,4,7)$&$15$&$(1,3,5,6)$&$15$&$(1,4,5,5)$&$15$&$(2,2,5,6)$&$15$&$(2,3,3,7)$&$15$\\
\oddrow $(2,3,4,6)$&$15$&$(2,3,5,5)$&$15$&$(3,3,4,5)$&$15$&$(1,1,6,8)$&$16$&$(1,2,3,10)$&$16$\\
\evnrow $(1,2,5,8)$&$16$&$(1,2,6,7)$&$16$&$(1,3,4,8)$&$16$&$(1,3,5,7)$&$16$&$(1,4,4,7)$&$16$\\
\oddrow $(1,4,5,6)$&$16$&$(2,2,5,7)$&$16$&$(2,3,3,8)$&$16$&$(2,3,4,7)$&$16$&$(2,3,5,6)$&$16$\\
\evnrow $(3,4,4,5)$&$16$&$(1,2,3,11)$&$17$&$(1,3,5,8)$&$17$&$(1,4,5,7)$&$17$&$(2,4,5,6)$&$17$\\
\oddrow $(1,1,7,9)$&$18$&$(1,2,3,12)$&$18$&$(1,2,6,9)$&$18$&$(1,2,7,8)$&$18$&$(1,3,4,10)$&$18$\\
\evnrow $(1,3,5,9)$&$18$&$(1,3,6,8)$&$18$&$(1,4,4,9)$&$18$&$(1,4,5,8)$&$18$&$(1,4,6,7)$&$18$\\
\oddrow $(1,5,6,6)$&$18$&$(2,2,5,9)$&$18$&$(2,3,4,9)$&$18$&$(2,3,5,8)$&$18$&$(2,3,6,7)$&$18$\\
\evnrow $(2,4,5,7)$&$18$&$(3,4,4,7)$&$18$&$(3,4,5,6)$&$18$&$(1,3,7,8)$&$19$&$(1,4,5,9)$&$19$\\
\oddrow $(1,5,6,7)$&$19$&$(2,3,5,9)$&$19$&$(1,1,8,10)$&$20$&$(1,2,5,12)$&$20$&$(1,2,7,10)$&$20$\\
\evnrow $(1,2,8,9)$&$20$&$(1,3,5,11)$&$20$&$(1,3,6,10)$&$20$&$(1,4,5,10)$&$20$&$(1,4,6,9)$&$20$\\
\oddrow $(1,5,6,8)$&$20$&$(2,3,5,10)$&$20$&$(2,3,7,8)$&$20$&$(2,4,5,9)$&$20$&$(2,5,5,8)$&$20$\\
\evnrow $(2,5,6,7)$&$20$&$(3,3,4,10)$&$20$&$(3,4,5,8)$&$20$&$(3,4,6,7)$&$20$&$(3,5,5,7)$&$20$\\
\oddrow $(4,4,5,7)$&$20$&$(4,5,5,6)$&$20$&$(1,3,4,13)$&$21$&$(1,3,7,10)$&$21$&$(1,4,6,10)$&$21$\\
\evnrow $(1,4,7,9)$&$21$&$(1,5,6,9)$&$21$&$(1,5,7,8)$&$21$&$(1,6,7,7)$&$21$&$(2,3,5,11)$&$21$\\
\oddrow $(2,3,7,9)$&$21$&$(2,5,6,8)$&$21$&$(2,5,7,7)$&$21$&$(3,4,5,9)$&$21$&$(3,4,7,7)$&$21$\\
\evnrow $(3,5,6,7)$&$21$&$(1,2,8,11)$&$22$&$(1,3,7,11)$&$22$&$(1,4,6,11)$&$22$&$(1,4,7,10)$&$22$\\
\oddrow $(1,6,7,8)$&$22$&$(2,2,7,11)$&$22$&$(2,3,6,11)$&$22$&$(2,3,7,10)$&$22$&$(2,4,5,11)$&$22$\\
\evnrow $(2,4,7,9)$&$22$&$(2,5,7,8)$&$22$&$(3,4,4,11)$&$22$&$(3,4,7,8)$&$22$&$(4,5,6,7)$&$22$\\
\oddrow $(1,4,7,11)$&$23$&$(1,5,6,11)$&$23$&$(2,3,7,11)$&$23$&$(2,5,7,9)$&$23$&$(3,4,5,11)$&$23$\\
\evnrow $(3,5,7,8)$&$23$&$(1,2,9,12)$&$24$&$(1,3,4,16)$&$24$&$(1,3,8,12)$&$24$&$(1,4,7,12)$&$24$\\
\oddrow $(1,4,8,11)$&$24$&$(1,4,9,10)$&$24$&$(1,5,6,12)$&$24$&$(1,6,7,10)$&$24$&$(1,6,8,9)$&$24$\\
\evnrow $(2,3,5,14)$&$24$&$(2,3,7,12)$&$24$&$(2,3,8,11)$&$24$&$(2,5,6,11)$&$24$&$(2,5,8,9)$&$24$\\
\oddrow $(2,6,7,9)$&$24$&$(3,4,5,12)$&$24$&$(3,4,6,11)$&$24$&$(3,4,8,9)$&$24$&$(3,5,7,9)$&$24$\\
\evnrow $(3,5,8,8)$&$24$&$(3,6,7,8)$&$24$&$(4,5,6,9)$&$24$&$(4,5,7,8)$&$24$&$(1,3,10,11)$&$25$\\
\oddrow $(1,5,7,12)$&$25$&$(1,6,7,11)$&$25$&$(1,6,8,10)$&$25$&$(1,7,8,9)$&$25$&$(2,5,7,11)$&$25$\\
\evnrow $(3,4,5,13)$&$25$&$(3,4,7,11)$&$25$&$(3,5,8,9)$&$25$&$(1,2,10,13)$&$26$&$(1,4,8,13)$&$26$\\
\oddrow $(1,5,7,13)$&$26$&$(2,3,8,13)$&$26$&$(2,4,7,13)$&$26$&$(2,5,6,13)$&$26$&$(2,5,7,12)$&$26$\\
\evnrow $(2,5,8,11)$&$26$&$(3,4,6,13)$&$26$&$(4,5,6,11)$&$26$&$(5,6,7,8)$&$26$&$(1,4,9,13)$&$27$\\
\oddrow $(1,5,9,12)$&$27$&$(1,6,7,13)$&$27$&$(1,7,9,10)$&$27$&$(2,3,5,17)$&$27$&$(2,3,9,13)$&$27$\\
\evnrow $(2,5,7,13)$&$27$&$(3,4,7,13)$&$27$&$(3,4,9,11)$&$27$&$(3,5,7,12)$&$27$&$(3,5,8,11)$&$27$\\
\oddrow $(3,7,8,9)$&$27$&$(4,5,7,11)$&$27$&$(4,6,7,10)$&$27$&$(1,2,11,14)$&$28$&$(1,3,10,14)$&$28$\\
\evnrow $(1,4,9,14)$&$28$&$(1,5,8,14)$&$28$&$(1,6,7,14)$&$28$&$(1,7,8,12)$&$28$&$(1,7,9,11)$&$28$\\
\oddrow $(1,8,9,10)$&$28$&$(2,3,7,16)$&$28$&$(2,3,9,14)$&$28$&$(2,5,7,14)$&$28$&$(2,5,8,13)$&$28$\\
\evnrow $(2,6,7,13)$&$28$&$(2,7,8,11)$&$28$&$(2,7,9,10)$&$28$&$(3,5,6,14)$&$28$&$(3,5,7,13)$&$28$\\
\oddrow $(4,5,6,13)$&$28$&$(4,5,7,12)$&$28$&$(2,7,9,11)$&$29$&$(3,5,8,13)$&$29$&$(1,2,12,15)$&$30$\\
\evnrow $(1,4,10,15)$&$30$&$(1,5,9,15)$&$30$&$(1,6,8,15)$&$30$&$(1,6,10,13)$&$30$&$(1,7,10,12)$&$30$\\
\oddrow $(1,8,10,11)$&$30$&$(2,3,5,20)$&$30$&$(2,3,10,15)$&$30$&$(2,4,9,15)$&$30$&$(2,5,8,15)$&$30$\\
\evnrow $(2,5,9,14)$&$30$&$(2,5,11,12)$&$30$&$(2,6,7,15)$&$30$&$(2,7,9,12)$&$30$&$(2,7,10,11)$&$30$\\
\oddrow $(3,4,5,18)$&$30$&$(3,4,8,15)$&$30$&$(3,5,7,15)$&$30$&$(3,5,8,14)$&$30$&$(3,5,10,12)$&$30$\\
\evnrow $(3,8,9,10)$&$30$&$(4,5,6,15)$&$30$&$(4,5,9,12)$&$30$&$(4,5,10,11)$&$30$&$(4,6,7,13)$&$30$\\
\oddrow $(5,6,9,10)$&$30$&$(6,7,8,9)$&$30$&$(3,7,10,11)$&$31$&$(5,6,7,13)$&$31$&$(1,6,9,16)$&$32$\\
\evnrow $(1,7,8,16)$&$32$&$(1,8,9,14)$&$32$&$(2,3,11,16)$&$32$&$(2,5,9,16)$&$32$&$(2,7,8,15)$&$32$\\
\oddrow $(3,5,8,16)$&$32$&$(3,8,10,11)$&$32$&$(4,5,7,16)$&$32$&$(4,5,9,14)$&$32$&$(4,7,10,11)$&$32$\\
\evnrow $(5,6,8,13)$&$32$&$(5,7,8,12)$&$32$&$(5,7,9,11)$&$32$&$(6,7,8,11)$&$32$&$(1,5,11,16)$&$33$\\
\oddrow $(1,6,11,15)$&$33$&$(1,8,11,13)$&$33$&$(1,9,11,12)$&$33$&$(2,7,11,13)$&$33$&$(3,4,11,15)$&$33$\\
\evnrow $(3,7,10,13)$&$33$&$(3,7,11,12)$&$33$&$(3,9,10,11)$&$33$&$(4,5,11,13)$&$33$&$(6,7,9,11)$&$33$\\
\oddrow $(1,3,13,17)$&$34$&$(1,5,11,17)$&$34$&$(1,6,10,17)$&$34$&$(1,7,9,17)$&$34$&$(2,4,11,17)$&$34$\\
\evnrow $(2,7,8,17)$&$34$&$(3,5,9,17)$&$34$&$(4,5,8,17)$&$34$&$(1,7,10,17)$&$35$&$(2,7,9,17)$&$35$\\
\oddrow $(2,7,11,15)$&$35$&$(3,5,13,14)$&$35$&$(3,8,10,14)$&$35$&$(4,5,7,19)$&$35$&$(4,7,9,15)$&$35$\\
\evnrow $(5,6,7,17)$&$35$&$(5,6,11,13)$&$35$&$(5,7,11,12)$&$35$&$(1,4,13,18)$&$36$&$(1,5,12,18)$&$36$\\
\oddrow $(1,6,11,18)$&$36$&$(1,7,10,18)$&$36$&$(1,8,9,18)$&$36$&$(1,9,10,16)$&$36$&$(1,9,12,14)$&$36$\\
\evnrow $(1,10,12,13)$&$36$&$(2,5,11,18)$&$36$&$(2,5,12,17)$&$36$&$(2,7,9,18)$&$36$&$(2,7,12,15)$&$36$\\
\oddrow $(2,8,9,17)$&$36$&$(2,9,11,14)$&$36$&$(2,9,12,13)$&$36$&$(3,4,11,18)$&$36$&$(3,7,8,18)$&$36$\\
\evnrow $(3,8,11,14)$&$36$&$(3,10,11,12)$&$36$&$(4,5,9,18)$&$36$&$(4,6,9,17)$&$36$&$(4,7,9,16)$&$36$\\
\oddrow $(4,9,10,13)$&$36$&$(4,9,11,12)$&$36$&$(5,6,7,18)$&$36$&$(5,6,9,16)$&$36$&$(5,8,9,14)$&$36$\\
\evnrow $(6,7,8,15)$&$36$&$(6,9,10,11)$&$36$&$(1,8,10,19)$&$38$&$(2,3,14,19)$&$38$&$(2,5,12,19)$&$38$\\
\oddrow $(2,7,10,19)$&$38$&$(2,8,9,19)$&$38$&$(4,6,9,19)$&$38$&$(1,6,13,19)$&$39$&$(1,7,13,18)$&$39$\\
\evnrow $(1,11,13,14)$&$39$&$(2,5,13,19)$&$39$&$(2,9,13,15)$&$39$&$(3,4,13,19)$&$39$&$(3,5,13,18)$&$39$\\
\oddrow $(3,7,13,16)$&$39$&$(3,8,13,15)$&$39$&$(3,11,12,13)$&$39$&$(4,5,13,17)$&$39$&$(4,7,13,15)$&$39$\\
\evnrow $(5,9,12,13)$&$39$&$(6,9,11,13)$&$39$&$(7,8,11,13)$&$39$&$(7,9,10,13)$&$39$&$(1,3,16,20)$&$40$\\
\oddrow $(1,6,13,20)$&$40$&$(1,7,12,20)$&$40$&$(1,8,11,20)$&$40$&$(1,10,13,16)$&$40$&$(2,5,13,20)$&$40$\\
\evnrow $(2,7,11,20)$&$40$&$(3,4,13,20)$&$40$&$(3,7,10,20)$&$40$&$(3,8,9,20)$&$40$&$(3,8,10,19)$&$40$\\
\oddrow $(3,10,11,16)$&$40$&$(3,10,13,14)$&$40$&$(4,7,9,20)$&$40$&$(4,7,10,19)$&$40$&$(5,7,8,20)$&$40$\\
\evnrow $(5,7,12,16)$&$40$&$(5,8,13,14)$&$40$&$(6,7,10,17)$&$40$&$(7,8,10,15)$&$40$&$(7,10,11,12)$&$40$\\
\oddrow $(8,9,10,13)$&$40$&$(1,6,14,21)$&$42$&$(1,8,12,21)$&$42$&$(1,9,11,21)$&$42$&$(1,12,14,15)$&$42$\\
\evnrow $(2,5,14,21)$&$42$&$(2,6,13,21)$&$42$&$(2,7,12,21)$&$42$&$(2,9,10,21)$&$42$&$(3,4,14,21)$&$42$\\
\oddrow $(3,5,13,21)$&$42$&$(3,7,11,21)$&$42$&$(3,8,10,21)$&$42$&$(3,8,14,17)$&$42$&$(3,12,13,14)$&$42$\\
\evnrow $(4,5,12,21)$&$42$&$(4,5,14,19)$&$42$&$(4,6,11,21)$&$42$&$(4,7,10,21)$&$42$&$(4,9,14,15)$&$42$\\
\oddrow $(5,6,14,17)$&$42$&$(5,7,9,21)$&$42$&$(5,11,12,14)$&$42$&$(6,7,8,21)$&$42$&$(6,9,13,14)$&$42$\\
\evnrow $(8,9,11,14)$&$42$&$(2,3,17,22)$&$44$&$(2,7,13,22)$&$44$&$(3,5,14,22)$&$44$&$(3,8,11,22)$&$44$\\
\oddrow $(3,11,14,16)$&$44$&$(4,7,11,22)$&$44$&$(4,11,14,15)$&$44$&$(6,7,9,22)$&$44$&$(6,8,11,19)$&$44$\\
\evnrow $(7,8,11,18)$&$44$&$(1,11,15,18)$&$45$&$(2,9,15,19)$&$45$&$(2,11,15,17)$&$45$&$(4,9,15,17)$&$45$\\
\oddrow $(5,9,14,17)$&$45$&$(6,11,13,15)$&$45$&$(7,9,10,19)$&$45$&$(8,9,13,15)$&$45$&$(1,7,15,23)$&$46$\\
\evnrow $(1,8,14,23)$&$46$&$(1,9,13,23)$&$46$&$(1,10,12,23)$&$46$&$(2,10,11,23)$&$46$&$(3,7,13,23)$&$46$\\
\oddrow $(4,5,14,23)$&$46$&$(4,9,10,23)$&$46$&$(5,7,11,23)$&$46$&$(6,7,10,23)$&$46$&$(1,7,16,24)$&$48$\\
\evnrow $(1,12,16,19)$&$48$&$(2,9,13,24)$&$48$&$(3,5,16,24)$&$48$&$(3,10,16,19)$&$48$&$(4,9,11,24)$&$48$\\
\oddrow $(5,6,13,24)$&$48$&$(5,6,16,21)$&$48$&$(5,8,11,24)$&$48$&$(6,11,15,16)$&$48$&$(7,12,13,16)$&$48$\\
\evnrow $(9,10,13,16)$&$48$&$(9,11,12,16)$&$48$&$(1,10,14,25)$&$50$&$(1,11,13,25)$&$50$&$(2,3,20,25)$&$50$\\
\oddrow $(2,7,16,25)$&$50$&$(2,9,14,25)$&$50$&$(2,11,12,25)$&$50$&$(3,8,14,25)$&$50$&$(4,10,11,25)$&$50$\\
\evnrow $(6,8,11,25)$&$50$&$(6,9,10,25)$&$50$&$(3,11,17,20)$&$51$&$(5,6,17,23)$&$51$&$(5,11,17,18)$&$51$\\
\oddrow $(9,11,14,17)$&$51$&$(1,8,17,26)$&$52$&$(1,9,16,26)$&$52$&$(2,7,17,26)$&$52$&$(3,4,19,26)$&$52$\\
\evnrow $(3,7,16,26)$&$52$&$(4,5,17,26)$&$52$&$(4,7,15,26)$&$52$&$(5,9,12,26)$&$52$&$(7,8,11,26)$&$52$\\
\oddrow $(7,9,10,26)$&$52$&$(1,8,18,27)$&$54$&$(1,12,14,27)$&$54$&$(2,7,18,27)$&$54$&$(2,12,13,27)$&$54$\\
\evnrow $(3,7,17,27)$&$54$&$(3,10,14,27)$&$54$&$(4,6,17,27)$&$54$&$(4,10,13,27)$&$54$&$(5,6,16,27)$&$54$\\
\oddrow $(5,8,14,27)$&$54$&$(6,8,13,27)$&$54$&$(6,10,11,27)$&$54$&$(7,8,12,27)$&$54$&$(1,11,16,28)$&$56$\\
\evnrow $(2,11,15,28)$&$56$&$(3,8,17,28)$&$56$&$(4,11,13,28)$&$56$&$(5,6,17,28)$&$56$&$(5,7,16,28)$&$56$\\
\oddrow $(5,11,12,28)$&$56$&$(7,8,13,28)$&$56$&$(8,9,11,28)$&$56$&$(3,10,16,29)$&$58$&$(4,7,18,29)$&$58$\\
\evnrow $(4,11,14,29)$&$58$&$(6,10,13,29)$&$58$&$(1,9,20,30)$&$60$&$(1,12,17,30)$&$60$&$(1,15,20,24)$&$60$\\
\oddrow $(3,7,20,30)$&$60$&$(3,8,19,30)$&$60$&$(3,11,16,30)$&$60$&$(4,9,17,30)$&$60$&$(6,15,19,20)$&$60$\\
\evnrow $(7,11,12,30)$&$60$&$(7,15,18,20)$&$60$&$(8,9,13,30)$&$60$&$(12,13,15,20)$&$60$&$(3,11,17,31)$&$62$\\
\oddrow $(5,7,19,31)$&$62$&$(5,9,17,31)$&$62$&$(5,12,14,31)$&$62$&$(6,11,14,31)$&$62$&$(7,11,13,31)$&$62$\\
\evnrow $(1,10,22,33)$&$66$&$(2,9,22,33)$&$66$&$(3,8,22,33)$&$66$&$(4,7,22,33)$&$66$&$(6,7,20,33)$&$66$\\
\oddrow $(6,13,14,33)$&$66$&$(4,11,19,34)$&$68$&$(5,8,21,34)$&$68$&$(5,13,16,34)$&$68$&$(8,11,15,34)$&$68$\\
\evnrow $(1,14,20,35)$&$70$&$(4,14,17,35)$&$70$&$(10,11,14,35)$&$70$&$(10,12,13,35)$&$70$&$(1,11,24,36)$&$72$\\
\oddrow $(5,7,24,36)$&$72$&$(7,13,16,36)$&$72$&$(8,11,17,36)$&$72$&$(1,12,26,39)$&$78$&$(2,11,26,39)$&$78$\\
\evnrow $(3,10,26,39)$&$78$&$(4,9,26,39)$&$78$&$(5,8,26,39)$&$78$&$(6,7,26,39)$&$78$&$(7,16,17,40)$&$80$\\
\oddrow $(3,11,28,42)$&$84$&$(5,9,28,42)$&$84$&$(4,11,30,45)$&$90$&$(7,18,20,45)$&$90$&$(10,17,18,45)$&$90$\\
\evnrow $(5,11,32,48)$&$96$&$(5,12,34,51)$&$102$&$(6,11,34,51)$&$102$&&&&\\
\end{longtable}
\end{center}
\bibliographystyle{amsplain}
\providecommand{\bysame}{\leavevmode\hbox to3em{\hrulefill}\thinspace}
\providecommand{\MR}{\relax\ifhmode\unskip\space\fi MR }
\providecommand{\MRhref}[2]{%
  \href{http://www.ams.org/mathscinet-getitem?mr=#1}{#2}
}
\providecommand{\href}[2]{#2}

\end{document}